\documentclass[a4paper,11pt]{amsart}
\usepackage[utf8]{inputenc}
\usepackage{amssymb, mathrsfs, amsfonts, amsmath}
\usepackage{mathtools} 
\usepackage{hyperref}
\usepackage{amsthm}
\usepackage{tikz}
\usepackage[english]{babel}
\usepackage{xcolor} 
\usepackage{geometry}
\usepackage{lipsum}

\title[General Maximal Restriction]{Low-dimensional maximal restriction principles for the Fourier transform}
\author{Jo\~ao P. G. Ramos}

\newcommand{\R}{\mathbb{R}}
\newcommand{\Z}{\mathbb{Z}}

\newcommand{\C}{\mathbb{C}}
\newcommand{\mmd}{\mathrm{d}}

\newtheorem{theorem}{Theorem}

\newtheorem{lemma}{Lemma}
\newtheorem{prop}{Proposition}

\newtheorem{quest}{Question}

\def\Xint#1{\mathchoice
{\XXint\displaystyle\textstyle{#1}}%
{\XXint\textstyle\scriptstyle{#1}}%
{\XXint\scriptstyle\scriptscriptstyle{#1}}%
{\XXint\scriptscriptstyle\scriptscriptstyle{#1}}%
\!\int}
\def\XXint#1#2#3{{\setbox0=\hbox{$#1{#2#3}{\int}$}
\vcenter{\hbox{$#2#3$}}\kern-.5\wd0}}

\def\dashint{\Xint-}
\begin{document}
\begin{abstract}
Following the ideas in \cite{Ramos1}, we prove abstract maximal results for the Fourier transform. Our results deal mainly with maximal operators of convolution-type and $r-$average maximal functions. 
As a by-product of our techniques we obtain spherical maximal restriction estimates, as well as restriction estimates for $2-$average maximal functions, answering thus points left open by V. Kova\v{c} \cite{Kovac1} and 
M\"uller, Ricci and Wright \cite{MRW}.
\end{abstract}
\maketitle

\section{Introduction} 

The classical \emph{restriction problem} for the Fourier transforms asks for the largest possible range of exponents $1 \le p,q \le + \infty$ so that
an inequality of the form 
\begin{equation}\label{restriction1}
\| \widehat{f}|_{S}\|_{L^q(S)} \le C_{p,q,d,S} \|f\|_{L^p(\R^d)}
\end{equation}
holds for any function $f \in \mathcal{S}(\R^d).$ Here, $S$ is taken to be a subset of $\R^d$, endowed with a suitable measure. \\

The existence of such \emph{a priori} inequalities allows one to define restrictions of Fourier transforms of $L^p$ functions to 
smaller sets in the $L^q-$sense. Recently, effort has been put into extending this definition to a \emph{pointwise} sense: one has to look instead at 
\[
\| \mathcal{M}(\widehat{f})\|_{L^q(S)} \le C_{p,q,d,S} \|f\|_{L^p(\R^d)},
\]
where $\mathcal{M}$ is a suitable maximal operator. In \cite{MRW}, the authors prove, for the first time, such a statement about restriction to curves. Their techniques adapt the ones in \cite{CarlSjoelin} to the maximal context. The works of Vitturi \cite{Vitturi}, Kova\v{c} and Oliveira e Silva \cite{KovacOliveira} and Ramos \cite{Ramos1}
have subsequently dealt with this problem, extending the maximal restriction property to higher dimensions, considering variational versions of it and sharpening the results in \cite{MRW}. \\

More recently, Kova\v{c} \cite{Kovac1} proved a general, abstract principle for such pointwise statements to hold. One of his results is that, whenever restriction estimates like \eqref{restriction1} hold with $p<q$, and whenever $\mu: \mathcal{B}(\R^d) \to \C$ is a complex measure such that $|\nabla \widehat{\mu}(\xi)| \le D(1+|\xi|)^{-1-\eta},$ for some $\eta >0,$ then 
\begin{equation}\label{kovactheo}
\| \sup_{t > 0} | \widehat{f} * \mu_t (x)| \|_{L^q(S)} \le C_{p,q,S,\eta} \cdot D \|f\|_{L^p(\R^d)}.
\end{equation}
Here, $\mu_t(E) := \mu(t^{-1}E).$ Note that $\mmd \mu = \chi_{B(0,1)}(x) \mmd x$ satisfies the Fourier decay condition above in \emph{any} dimension, which generalizes the results of Vitturi \cite{Vitturi}, M\"uller, Ricci, Wright \cite{MRW} and Kova\v{c} and Oliveira e Silva \cite{KovacOliveira}.\\ 

The purpose of this note is to employ the techniques in \cite{Ramos1} to extend inequality \eqref{kovactheo} in low-dimensional cases not covered by Kova\v{c}'s techniques. Additionally, we simplify the techniques in \cite{Ramos1} in order to extend a result from \cite{MRW}. 

\subsection{Two-dimensional results}\label{twodimm} In \eqref{kovactheo}, the main requirement on the measure $\mu$ that $|\nabla \widehat{\mu}(\xi)| \lesssim_{\eta,\mu} (1+|\xi|)^{-1-\eta},$ for some $\eta >0,$ is 
only satisfied  by the \emph{spherical measure} $\mmd \mu = \mmd \sigma_{\mathbb{S}^{d-1}}$ if $d \ge 4.$ Therefore, Kova\v{c}'s result does not yield bounds for lower-dimensional restrictions of spherical maximal functions of the Fourier transform.
This was our motivation for the first result of this paper. 

\begin{theorem}\label{meastheo1}
Let $\mu$ be a positive, finite Borel measure defined in $\R^2$, and suppose that the maximal function 
\[
M_{\mu} g(x) := \sup_{t > 0} |g|*\mu_t(x).
\]
is bounded from $L^r(\R^2) \to L^r(\R^2),$ whenever $r > 2.$ Then the following bound holds: 
\[ 
\|M_{\mu}(\widehat{f})\|_{L^q(\mathbb{S}^1)} \le C_{p,\mu} \|f\|_{L^p(\R^2)},
\]
where $1 \le p < \frac{4}{3}, p' \ge 3q.$ 
\end{theorem}
In Proposition \ref{multipl} at the end of this note, we prove that Kova\v{c}'s \cite{Kovac1} assumptions on the measure \emph{imply} ours. The spherical maximal function in dimensions 2, 3 is an example that shows, as elaborated in Section \ref{multp}, that Theorems \ref{meastheo1} and \ref{meastheo2} are strictly stronger. \\

On the other hand, in \cite{MRW}, the authors, in the end of their manuscript, make use of the maximal function 
\[
M_2(h) := M(|h|^2)^{1/2},
\]
where $Mf(x) = \sup_{r>0} \dashint_{B(x,r)} |f|$ denotes the usual Hardy--Littlewood maximal function, 
to prove results about Lebesgue points of the Fourier transform on curves in the range $1 \le p < \frac{8}{7}$. In \cite{Ramos1}, this author circumvents this problem by considering a suitable linearization instead of working with $M_2.$ Our next result
combines the two approaches: 

\begin{theorem}\label{rtheo1} Let $1\le r \le 2.$ Define the maximal functions $M_rh(x) := (M(|h|^r)(x))^{1/r}.$ The following bound holds: 
\[ 
\|M_r(\widehat{f})\|_{L^q(\mathbb{S}^1)} \le C_{p,r} \|f\|_{L^p(\R^2)},
\]
where $1 \le p < \frac{4}{3}, p' \ge 3q.$ 
\end{theorem} 

The main feature in the proofs of these Theorems is the linearization method employed in \cite{Ramos1} 
together with Lemmata \ref{mainlemma} and \ref{mainlemma2}. These, on the other hand, provide a way to bypass the interpolation scheme employed in \cite[Lemmata~1~and~2]{Ramos1}. Also, in the case where one takes $\mmd \mu = \mmd \sigma_{\mathbb{S}^1}$ to be the arc-length measure in the circle
the interpolation idea fails due to the lack of $L^2(\R^2)$ bounds for maximal functions, whereas working directly with the aid of the Hausdorff--Young inequality gives us the result, as long as the measure we consider satisfies the above conditions. 
By the celebrated result of Bourgain \cite{Bourgain1}, this is \emph{exactly} the case for the circular maximal function in dimension 2. \\ 

In Section \ref{knapp}, we present two different kinds of counterexamples, in order to impose restrictions on $r$ so that Theorem \ref{rtheo1} can hold. Both the examples yield the same $r \le 4$ bound, whereas Theorem \ref{rtheo1} only works in the $r \le 2$ case. 
One is led to pose the following question:

\begin{quest}\label{rquest1} Can the two-dimensional full range of maximal restriction inequalities hold for $M_{s}, 2< s \le 4?$ 
\end{quest} 

\subsection{Three-dimensional results} In dimension 3, our main theorems deal with the Tomas-Stein exponent case, in both the context of measures as well as in the context of $M_r-$maximal functions: 

\begin{theorem}\label{meastheo2} 
Let Let $\mu$ be a positive, finite Borel measure defined in $\R^3$, and suppose that the maximal function 
\[
M_{\mu} g(x) := \sup_{t > 0} |g|*\mu_t(x).
\]
is bounded from $L^2(\R^3) \to L^2(\R^3).$ Then the following bound holds: 
\[ 
\|M_{\mu}(\widehat{f})\|_{L^2(\mathbb{S}^2)} \le C_{p,\mu} \|f\|_{L^{p}(\R^3)},
\]
where $1 \le p \le \frac{4}{3}.$ 
\end{theorem}

\begin{theorem}\label{rtheo2} Let $1 \le r < 2.$ Then the following bound holds: 
\[ 
\|M_r(\widehat{f})\|_{L^2(\mathbb{S}^2)} \le C_{p,r} \|f\|_{L^{p}(\R^3)},
\]
where $1 \le p \le \frac{4}{3}.$ Aditionally, the quadratic maximal function $M_2$ satisfies that 
\[
\|M_2(\widehat{f})\|_{L^2(\mathbb{S}^2)} \le C_{p,r} \|f\|_{L^{p}(\R^3)},
\]
whenever $1 \le p < \frac{4}{3}.$
\end{theorem}

We prove these results in Section \ref{threedim} by merging the ideas in Theorems \ref{meastheo1} and \ref{rtheo1} with Vitturi's method. As a by-product, the counterexamples built in Section \ref{twodimm} provide us with the restriction that 
$s \le 2$ in order for Theorem \ref{rtheo2} to hold. In particular, a further use of one of these counterexamples in higher dimensions gives us as a direct corollary that the only dimensions in which a full-range restriction result for the \emph{strong} maximal function 
\[
M_{\mathcal{S}} f(x) := \sup_{\substack{R \text{ axis parallel}, \\ \text{centered at } x}} \dashint_R |f| 
\]
of the Fourier transform could hold are $d=2,3$. We talk about this property in more detail in Proposition \ref{sharpp}. 

\subsection{Notation} 
In what follows, we denote $A \lesssim B$ to mean that $A \le C \cdot B,$ for some universal constant $C>0$. If we let $C$ depend on a parameter $\alpha,$ we write $A \lesssim_{\alpha} B.$ We suppress this notation in case the specific dependence on $\alpha$ is not important. 
We also normalize the Fourier transform as $\mathcal{F}f(\xi) = \widehat{f}(\xi) = \int_{\R^d} e^{-2 \pi i x \cdot \xi} f(x) \, \mmd x.$ Finally, we often write $\dashint_B g := \frac{1}{|B|} \int_B g$ for the average of $g$ over a set $B$. \\

\noindent \textbf{Acknowledgements.} The author would like to thank Felipe Gon\c calves for helpful discussions which led to the proof of Theorem \ref{meastheo1} and to a great deal of inspiration for the other results in this manuscript. He is also indebted to
his supervisor Prof. Dr. Christoph Thiele for reading this manuscript and giving advice on how to improve the presentation. The author also acknowledges finantial support
by the Deutscher Akademischer Austauschdient (DAAD). 

\section{Proof of Theorems \ref{meastheo1} and \ref{rtheo1}}\label{twodim}

\subsection{Proof of Theorem \ref{meastheo1}} The basic outline of the proof is essentially the same as in the proof of \cite[Theorem~1]{Ramos1}. After using the Kolmogorov--Saliverstov linearization method and letting $g(z) = \frac{\overline{\widehat{f}(z)}}{|\widehat{f}(z)|},$ it suffices to prove bounds for 
\[ 
M_{\mu,g,t(\cdot)}f(x) = \int_{\R^2} \widehat{f}(x-y) g(x-y) \mmd \mu_{t(x)} (y).
\]
Here, we actually regard $M_{\mu,g,t(\cdot)}$ as an operator with a \emph{fixed} $g$, prove bounds for it and then substitute the chosen $g$ above. An application of Plancherel's Theorem implies that 
\[
M_{\mu,g,t(\cdot)}f(x) = \int_{\R^2} f(\xi) e^{2\pi i x \cdot \xi} \widehat{A_x}(\xi) \mmd \xi,
\]
where $\mmd A_x(y) := g(x-y)\, \mmd \mu_{t(x)}(y).$ A dualization argument then implies that Theorem \ref{meastheo1} is equivalent to proving 
\[ 
M_{\mu,g,t(\cdot)}^{*} h(\xi) = \int_{\mathbb{S}^1} h(\omega) e^{-2\pi i \xi \cdot \omega} \widehat{A_{\omega}}(\xi) \mmd \sigma_{\mathbb{S}^1}(\omega)
\]
to be bounded from $L^{q'}(\mathbb{S}^1) \to L^{p'}(\R^2).$ Just like in the proof of \cite[Lemma~2]{Ramos1}, we write $\|M_{\mu,g,t(\cdot)}^{*} h\|_{p'} = 
\| (M_{\mu,g,t(\cdot)}^{*} h)^2\|_{p'/2}^{1/2}.$ Expanding the square gives
\[ 
(M_{\mu,g,t(\cdot)}^{*} h)^2(\xi) = \int_{\mathbb{S}^1 \times \mathbb{S}^1} h(\omega) \, h(\omega') \, e^{- 2 \pi i (\omega+\omega') \cdot \xi} \widehat{A_{\omega}}(\xi) \widehat{A_{\omega'}}(\xi) \,\mmd \sigma_{\mathbb{S}^1}(\omega) \, \mmd \sigma_{\mathbb{S}^1}(\omega'). 
\]
We perform two changes of variable: first, we parametrise the circle by $z(r) = (\cos(2\pi r),\sin(2\pi r)).$ After that, we take a pair of points $(t,s),t>s,$ into the point $x := z(t) + z(s).$ This map is easily seen to be a bijection from 
\[
\Delta := \{ (t,s) \in [0,1)^2, t> s\} \text{ to } B_2(0) \subset \R^2.
\]
After a calculation, we rewrite our operator as 
\[
(M_{\mu,g,t(\cdot)}^{*} h)^2(\xi) = 2 \int_{B_2(0)} H(x) e^{- 2 \pi i x \cdot \xi} \widehat{B_x}(\xi) \,\mmd x, 
\]
where 
\begin{equation}\label{fourier}
\widehat{B_x}(\xi) := \widehat{A_{z(t)}}(\xi) \widehat{A_{z(s)}}(\xi), 
\end{equation}
\[
H(x) := \frac{h(z(s))h(z(t))}{|\det(z'(s),z'(t))|} = \frac{h(z(s))h(z(t))}{4\pi^2|\sin(2\pi(s-t))|}.
\]
Notice that the factor 2 multiplying the integral comes from considering twice the contribution from the upper triangle. The representation for our squared operator leads us to our main Lemma, which is a generalization of \cite[Lemma~2]{Ramos1}:

\begin{lemma}\label{mainlemma} Let, for every $x \in \R^2,$ $B_x = \mu_{t_1(x)} * \cdots \mu_{t_k(x)}$ be  the convolution product of dilates $\mu_{t_1(x)},...,\mu_{t_k(x)}$ of a finite Borel measure such that
\begin{equation}\label{maxim}
\|M_{\mu} \|_{r \to r} < + \infty, \, \forall r >2.
\end{equation}
Assume, in addition, that the map $x \mapsto B_x$ is in $L^{\infty}(M(\R^2)),$ where $M(\R^2)$ denotes the space of finite Borel measures on $\R^2.$ If  
\[
Tf(\xi) = \int_{\R^2} \widehat{B_x}(\xi) e^{-2 \pi i x \cdot \xi} f(x) \,\mmd x, 
\]
then $T$ maps $L^p(\R^2)$ to $L^{p'}(\R^2)$ boundedly for $1 \le p <2.$ 
\end{lemma} 

\begin{proof} We write, for an arbitrary function $g \in L^1(\R^2) \cap L^2(\R^2),$ 
\[
\langle Tf,g \rangle = \int_{\R^2} \overline{g}(\xi) \int_{\R^2} \widehat{B_x}(\xi) e^{-2 \pi i x \cdot \xi} f(x) \,\mmd x \, \mmd \xi.
\]
By Fubini and Plancherel, this equals, in turn, 
\[
\int_{\R^2} f(x) \widehat{\overline{g}} * B_x(x) \mmd x. 
\]
By the definition of $B_x$, property \eqref{maxim} and the Hausdorff-Young inequality, we bound the absolute value of the integral above by 
\[ 
\int_{\R^2} |f(x)| |M_{\mu}^k (\widehat{\overline{g}})|(x) \, \mmd x \le \|f\|_{p} \||M_{\mu}^k (\widehat{\overline{g}})|\|_{p'} \le (C_{\mu})^k \|f\|_p \|\widehat{\overline{g}}\|_{p'} \le (C_{\mu})^k \|f\|_p\|g\|_{p}.
\]
This proves the asserted bound for $T.$
\end{proof}

Notice that the function $B_x$ in \eqref{fourier} satisfies the hypotheses of Lemma \ref{mainlemma}. Notice also that $p'/2 >2$. After applying the Lemma above we are left with 
\[
\| (M_{\mu,g,t(\cdot)}^{*} h)^2\|_{p'/2}^{1/2} \lesssim \|H\|_{L^{(p'/2)'}(B_2(0))}^{1/2}.
\]
To conclude the proof, we revert from $H$ back to a product to estimate the right-hand-side for $1\le (p'/2)' =: \eta <2:$ 

\begin{align}\label{estimimp}
\int_{B_2(0)} |H(x)|^{\eta} \, \mmd x= & \int_{\Delta} |h(z(t))h(z(s))|^{\eta} \cdot (4\pi^2|\sin(2\pi(s-t))|)^{1-\eta} \, \mmd t \mmd s  \cr
& \le C_p\||f|^{\eta}\|^{2}_{\frac{2}{3-\eta}} = C_p \|f\|_{\frac{2\eta}{3-\eta}}^{2\eta} = C_p \|f\|_{(p'/3)'}.\cr 
\end{align}

Here, the last inequality follows from the Hardy--Littlewood--Sobolev inequality for fractional integrals. Indeed, we can bound 

$$ {4\pi^2|\sin(2\pi(s-t))|}^{1-\eta} \lesssim \sum_{j=-2}^{2} |t-s-j|^{1-\eta},$$ 
and then notice that each summand on the right hand side leads to a translated fractional integral. The result follows for the range $1 \le p < \frac{4}{3}, p' \ge 3q$ by interpolating this bound
with the $L^1(\R^2) \to L^{\infty}(\mathbb{S}^1)$ bound, which follows in turn from the Riemann-Lebesgue Lemma and finiteness of the measure $\mu. \,\,\square$ 

\subsection{Proof of Theorem \ref{rtheo1}} In the same spirit as above, proving Theorem \ref{rtheo1} is equivalent to proving bounds for 

\[
M_{r,g,t(\cdot)} f(x) := \int_{\R^2} \widehat{f}(x-y) g_x(x-y) \chi_{t(x)}(y)\mmd y,
\]
where we will take, in the aftermath, 
\[
g_x(z) = \frac{\overline{\widehat{f}(z)} |\widehat{f}(z)|^{r-2}}{|B_{t(x)}(0)|^{1/r - 1} \cdot \|\widehat{f}\|_{L^r(B_{t(x)}(x))}^{r-1}}.
\]
With the above choice, the integral defining $M_{r,g,t(\cdot)}$ equals $\left(\int_{B_{t(x)}(0)} |\widehat{f}(x-y)|^r \mmd y\right)^{1/r}.$ We denote a $L^1-$normalized dilation of characteristic
function of the unit ball as $\chi_{a}(x) := (1/a^2) \cdot \chi(x/a).$ We then write the adjoint as 
\[
M_{r,g,t(\cdot)}^{*} h(\xi) = \int_{\mathbb{S}^1} h(\omega) e^{-2 \pi i \omega \cdot \xi} \widehat{\mathcal{A}_{\omega}}(\xi) \mmd \sigma_{\mathbb{S}^1}(\omega),
\]
with $\mathcal{A}_x(y) = g_x(x-y) \chi_{t(x)}(y).$ As before, we calculate $(M_{r,g,t(\cdot)}^{*})^2$ and change variables. It suffices to bound 
\begin{equation}\label{reqq}
(M_{r,g,t(\cdot)h}^{*})^2(\xi) = 2 \int_{B_2(0)} H(x) e^{- 2 \pi i x \cdot \xi} \widehat{\mathcal{B}_x}(\xi) \,\mmd x =: T_r H(\xi), 
\end{equation}
where, again, 
\[
H(x) := \frac{h(z(s))h(z(t))}{|\det(z'(s),z'(t))|} = \frac{h(z(s))h(z(t))}{4\pi^2|\sin(2\pi(s-t))|},
\]
and 
\begin{equation}\label{convdef}
\widehat{\mathcal{B}_x}(\xi) := \widehat{\mathcal{A}_{z(t)}}(\xi) \widehat{\mathcal{A}_{z(s)}}(\xi).
\end{equation}
Of course, $z(s) + z(t) = x.$ The next Lemma is the main tool for bounding \eqref{reqq}, in order to employ the previous techniques:

\begin{lemma}\label{mainlemma2} Let $u \in \mathcal{S}(\R^2).$ Suppose that we are given a measurable function $\mathcal{A} : \mathbb{R}^2 \to L^{r'}(\R^2)$ so that $\sup_{x \in \mathbb{R}^2} |B_x|^{1/r}\|\mathcal{A}_x\|_{L^{r'}} < + \infty$ and
$\text{support} (\mathcal{A}_x) \subset B_x$ for some ball $B_x$ centered at the origin. If we define $\mathcal{B}_x$ as in 
equation \eqref{convdef}, then it holds that
\[
u * \mathcal{B}_{x}(\theta) \le C \cdot M_r(M_r u)(\theta), \forall \theta \in \R^2,
\]
where $C$ is \emph{independent} of $x \in B_2(0).$ 
\end{lemma}

\begin{proof} We denote first $\pi_1(x), \pi_2(x) \in \mathbb{S}^1$ the points such that $\pi_1(x) + \pi_2(x) = x.$ The above convolution is 
\[
u * \mathcal{A}_{\pi_1(x)} * \mathcal{A}_{\pi_2(x)} (\theta).
\]
It suffices to prove that $u * \mathcal{A}_{\pi_1(x)} (\eta) \le C \cdot M_ru(\eta),$ as the same argument holds for the convolution with $\mathcal{A}_{\pi_2(x)}.$ 
We write 
\begin{align*}  & u * \mathcal{A}_{\pi_1(x)}(\eta)  = \int_{B_x} u(\eta - s) \mathcal{A}_{\pi_1(x)}(s) \mmd s \le (\sup_{z \in \R^2} \|\mathcal{A}_z\|_{L^{r'}}) \|u(\eta - \cdot)\|_{L^r(B_x)} \cr 
 & \lesssim |B_x|^{-1/r}  \|u(\eta - \cdot)\|_{L^r(B_x)} \le M_ru(\eta), \cr
\end{align*} 
where we have used H\"older's inequality and the properties of $\mathcal{A}.$ 
\end{proof}
With Lemma \ref{mainlemma2}, we are set to employ the techniques of the proof of Lemma \ref{mainlemma}. In fact, we let $G \in L^1(\R^2) \cap L^2(\R^2)$, and take $\mathcal{B}_x$ as defined in equation \eqref{convdef} with $\mathcal{A}_x(y) = g_x(x-y) \chi_{t(x)}(y).$ By a direct computation -- 
due to the dualization nature of our choice -- to check that this $\mathcal{A}$ satisfies the hypotheses of Lemma \ref{mainlemma2}. Therefore, we estimate the pairing:
\begin{align*}
 & \langle T_r H, G \rangle  = \int_{\R^2} \overline{G}(\xi) \left( \int_{B_2(0)} H(x) e^{- 2 \pi i x \cdot \xi} \widehat{\mathcal{B}_x}(\xi) \,\mmd x \right)  \, \mmd \xi \cr
 & = \int_{B_2(0)} H(x) \cdot \widehat{\overline{G}} * \mathcal{B}_x (x) \, \mmd x  \le \int_{B_2(0)} |H(x)| \cdot M_r (M_r \widehat{\overline{G}})(x) \mmd x \cr 
 & \le \|H\|_{L^p(B_2(0))} \|M_r(M_r \widehat{\overline{G}})\|_{p'} \le (C_r)^2 \|\widehat{\overline{G}}\|_{p'} \|H\|_{L^p(B_2(0))} \le \tilde{C}_{r,p} \|G\|_p  \|H\|_{L^p(B_2(0))}.
\end{align*}
We have, similarly as before, used Fubini and Plancherel Theorems together with Lemma \ref{mainlemma2} in the second line, and H\"older's inequality in combination with boundedness of $M_r$ in $L^{p'}$ (as $p' > 2 \ge r$) and the 
Hausdorff--Young inequality. \\ 

We conclude, by density, that $\|T_r H\|_{p'} \le \tilde{C}_{r,p} \|H\|_p, \, 1 \le p <2.$  Now one resumes from the calculation in \eqref{estimimp}, and our previous 
considerations allow us to finish, once one notices that the $L^1(\R^2) \to L^{\infty} (\mathbb{S}^1)$ boundedness in this case is also a direct consequence of the Riemann-Lebesgue lemma. $\,\,\square$  

\section{Proof of Theorems \ref{meastheo2} and \ref{rtheo2}}\label{threedim}  

\subsection{Proof of Theorem \ref{meastheo2}} The strategy here is a modification of the scheme of proof in \cite{Vitturi}. There, one uses an integral representation for the convolution of Fourier transforms. 
Here, as we are working with measures and not functions, such a representation only becomes available to some measures through delta calculus. We bypass this difficulty by
an argument similar to the one in the proofs of Theorems \ref{meastheo1} and \ref{rtheo1}. \\ 

Explicitly, we start by linearizing our operator through
\[
\mathcal{M}_{\mu,g,t(\cdot)} f(x) = \int_{\R^3} f(\xi) e^{2\pi i x \cdot \xi} \widehat{S_x}(\xi) \, \mmd \sigma (x),
\]
where $\mmd S_x (y) = g(x-y)\, \mmd \mu_{t(x)}(y), \|g\|_{\infty} \le 1.$ Again, we will take $g(z) = \frac{\overline{\widehat{f}(z)}}{|\widehat{f}(z)|}$ afterwards. The desired inequality translates into proving that 
\[
\|\mathcal{M}_{\mu,g,t(\cdot)} f\|_{L^4(\R^3)} \lesssim \|f\|_{L^2(\mathbb{S}^2)}. 
\]
We write the $L^4-$norm above as $\|(\mathcal{M}_{\mu,g,t(\cdot)} f)^2\|_2^{1/2},$ and evaluate the $L^2-$norm by duality: for any $h \in L^2(\R^3), \,\|h\|_2 \le 1,$ 
we have 
\begin{align}\label{eqqq}
 & \langle (\mathcal{M}_{\mu,g,t(\cdot)} f)^2 , h \rangle = \cr 
 & = \int_{\R^3} \left(\int_{(\mathbb{S}^2 )^2} f(x_1) f(x_2) e^{-2\pi i (x_1 + x_2) \cdot \xi} \widehat{S_{x_1}}(\xi) \widehat{S_{x_2}}(\xi) \, \mmd \sigma (x_1) \, \mmd \sigma (x_2)\right) h(\xi) \, \mmd \xi \cr 
 & = \int_{\mathbb{S}^2 \times \mathbb{S}^2} f(x_1) f(x_2) \left(\int_{\R^3} h(\xi) e^{-2\pi i (x_1 + x_2) \cdot \xi} \widehat{S_{x_1}}(\xi) \widehat{S_{x_2}}(\xi) \, \mmd \xi \right) \, \mmd \sigma(x_1) \mmd \sigma(x_2),\cr  
\end{align}
where we used Fubini's theorem to exchange integrals. Another application of Fubini's theorem in the innermost integral gives us that 
\begin{align*} 
 & \int_{\R^3 \times \R^3} g(x_1-y_1) g(x_2 - y_2) \widehat{h}((x_1 + x_2) - (y_1 + y_2))\,  \mmd \mu_{t(x_1)}(y_1) \, \mmd \mu_{t(x_2)}(y_2) = \cr
 & = \int_{\R^3} h(\xi) e^{-2 \pi i (x_1 + x_2) \cdot \xi} \widehat{S_{x_1}}(\xi) \widehat{S_{x_2}}(\xi) \, \mmd \xi.
\end{align*}
It is relatively simple to bound this integral: the integrand is pointwise bounded by 
\begin{align*}
 \int_{\R^3 \times \R^3} |\widehat{h}((x_1 + x_2) - (y_1 + y_2))| \,  \mmd \mu_{t(x_1)}(y_1) \, \mmd \mu_{t(x_2)}(y_2) & \le M_{\mu}(M_{\mu})(\widehat{h})(x_2 + x_1), \cr  
\end{align*}
where we used the definition of our maximal function associated to $\mu$. Thus, the integral we wish to estimate is bounded by 
\[
\int_{\mathbb{S}^2 \times \mathbb{S}^2} |f(x_1)| |f(x_2)| M_{\mu}(M_{\mu})(\widehat{h})(x_2 + x_1) \, \mmd \sigma(x_1) \, \mmd \sigma(x_2).
\]
By the Tomas-Stein theorem in dimension 3, as stated in \cite[Equation~2.3]{Vitturi}, the quantity above is at most a constant times 
\[
\|f\|_{L^2(\mathbb{S}^2)}^2 \|(M_{\mu})^2(\widehat{h})\|_{L^2(\R^3)} \le (C_{\mu})^2 \|f\|_{L^2(\mathbb{S}^2)}^2.
\]
Along with the previous considerations, it is exactly what we wanted to prove. $\square$

\subsection{Proof of Theorem \ref{rtheo2}} The general idea here is similar to the proofs above, so we move somewhat faster through it. In fact, we consider the maximal operator $M_2$ first. Like before, we define the linearization of this operator as 
\[
M_{2,g,t(\cdot)} f(x) := \int_{\R^3} \widehat{f}(x-y) g_x(x-y) \chi_{t(x)}(y)\mmd y,
\]
where, in the end, $\tilde{g}_x$ is to be taken as 
\[
\tilde{g}_x(z) = \frac{\overline{\widehat{f}(z)}}{|B_{t(x)}(0)|^{-1/2} \cdot \|\widehat{f}\|_{L^2(B_{t(x)}(x))}}.
\]
Like in the cases before, we fix $\tilde{g}_x$ with certain properties and then substitute the above to get our results. The formal adjoint of this operator is given by 
\[
M_{2,g,t(\cdot)}^{*} h(\xi) = \int_{\mathbb{S}^2} h(\omega) e^{-2 \pi i \omega \cdot \xi} \widehat{\mathcal{S}_{\omega}}(\xi) \mmd \sigma_{\mathbb{S}^2}(\omega),
\]
with $\mathcal{S}_x(y) = \tilde{g}_x(x-y) \chi_{t(x)}(y).$ This leads us to estimate, as before, the inner product $\langle (M_{2,g,t(\cdot)}^{*} h)^2, F \rangle.$ The calculation 
is entirely analogous to the one in \eqref{eqqq}, and we are led to estimate the function 
\[
\int_{\R^3} F(\xi) e^{-2\pi i (\omega_1 + \omega_2) \cdot \xi} \widehat{\mathcal{S}_{\omega_1}}(\xi) \widehat{\mathcal{S}_{\omega_2}}(\xi) \, \mmd \xi.  
\]
An application of Fubini's theorem, along with the calculations from the proofs of Theorems \ref{rtheo1} and \ref{meastheo2} yield \emph{pointwise} bounds for this integral by the iterated 
maximal function $M_2(M_2(\widehat{F}))(\omega_1 + \omega_2).$ This summarizes as 
\begin{equation}\label{maxx2}
|\langle (M_{2,g,t(\cdot)}^{*} h)^2, F \rangle| \le \int_{\mathbb{S}^2 \times \mathbb{S}^2} |h(\omega_1) h(\omega_2)| M_2(M_2(\widehat{F}))(\omega_1 + \omega_2) \, \mmd \sigma (\omega_1) \, \mmd \sigma (\omega_2). 
\end{equation}
In order to finish, we need to apply the following Lemma:
\begin{lemma}\label{tomasstein} Let $2 \le p \le \infty.$ There is a constant $C = C(p)$ such that, for all $v \in L^2(\mathbb{S}^2)$ and $W \in L^p(\R^3),$ it holds that
\[
\left| \int_{\mathbb{S}^2 \times \mathbb{S}^2} v(\omega_1) v(\omega_2) W(\omega_1 + \omega_2) \, \mmd \sigma(\omega_1) \, \mmd \sigma (\omega_2) \right| \le C \|v\|_{L^2(\mathbb{S}^2)}^2 \|W\|_{L^p(\R^3)}. 
\]
\end{lemma} 

\begin{proof} We define the operator
\[
T_{v_1}W(\omega_1) = \int_{\mathbb{S}^2} v_1(\omega_2) W(\omega_1 + \omega_2) \, \mmd \sigma (\omega_2)
\]
and note it satisfies the two following estimates:
\begin{itemize} 
\item For $p = \infty,$ the estimate $\|T_{v_1}W\|_{L^2(\mathbb{S}^2)} \lesssim \|v_1\|_{L^2(\mathbb{S}^2)} \|W\|_{\infty}$ follows by duality and triangle and H\"older's inequality. 
\item For $p=2,$ the estimate $\|T_{v_1}W\|_{L^2(\mathbb{S}^2)} \lesssim \|v_1\|_{L^2(\mathbb{S}^2)}\|W\|_{2}$ follows from the Tomas-Stein restriction theorem (see, e.g., \cite{Tomas, Foschi}), as stated in \cite{Vitturi}. In fact, for any 
two $v_1,v_2,$ we have
\[
\| (v_1 \, \mmd \sigma) * (v_2 \, \mmd \sigma) \|_2 = \| \widehat{ (v_1 \, \mmd \sigma)} \widehat{ (v_1 \, \mmd \sigma)} \|_2 \le \| \widehat{ (v_1 \, \mmd \sigma)}\|_4 \| \widehat{ (v_2 \, \mmd \sigma)}\|_4 \lesssim \|v_1\|_{L^2(\mathbb{S}^2)} \|v_2\|_{L^2(\mathbb{S}^2)}.
\]
The asserted inequality follows then by duality. 
\end{itemize}
The considerations above show that $T_{v_1}$ satisfies $L^{\infty}(\R^3) \to L^2(\mathbb{S}^2)$ and $L^2(\R^3) \to L^2(\mathbb{S}^2)$ estimates. By interpolation, it must also satisfy 
$L^p(\R^3) \to L^2(\mathbb{S}^2)$ estimates, with norm at most $\lesssim \|v_1\|_{L^2(\mathbb{S}^2)}$. By duality, this assertion is equivalent to 
\[
\left| \int_{\mathbb{S}^2 \times \mathbb{S}^2} v_1(\omega_1) v_2(\omega_2) W(\omega_1 + \omega_2) \, \mmd \sigma(\omega_1) \, \mmd \sigma (\omega_2) \right| \le C \|v_1\|_{L^2(\mathbb{S}^2)} \|v_2\|_{L^2(\mathbb{S}^2)} \|W\|_{L^p(\R^3)}. 
\]
By setting $v_1 = v_2$ one obtains the Lemma.
\end{proof} 
To finish the proof, we apply Lemma \ref{tomasstein} in \eqref{maxx2} with $\eta>2.$ Using that $M_2$ is bounded in $L^{\eta}$ and the Hausdorff--Young inequality gives 
\[
|\langle (M_{2,g,t(\cdot)}^{*} h)^2, F \rangle| \lesssim \|h\|^2_{L^2(\mathbb{S}^2)} \|\widehat{F}\|_{\eta} \le  \|h\|^2_{L^2(\mathbb{S}^2)} \|F\|_{\eta'}.
\]
It is straightforward to check that this last inequality is equivalent to $M_{2,g,t(\cdot)}^{*} h$ being bounded from $L^2$ to $L^{2\eta}.$ As $\eta>2$ was arbitrary, we finish this part of the proof. \\

In order to deal with $1 \le r <2,$ we use the pointwise domination $M_rf \le M_2f, 1 \le r \le 2.$ Thus the only missing point in the proof above is the endpoint $(\frac{4}{3},2).$ A combination of the proofs of Theorems
\ref{rtheo1} and \ref{meastheo2} gives us estimates in the endpoint case, in the same spirit as above. This time, the application of Lemma \ref{tomasstein} might be circumvented, as 
$M_r$ is bounded in $L^2.$ We skip the details. $\,\,\square$ 

\section{Comments, generalizations and remarks}\label{commsec}

\subsection{Maximal operators of convolution-type and multiplier theorems}\label{multp} Theorems \ref{meastheo1} and \ref{rtheo1} deal with maximal functions related to a measure $\mmd \mu.$ There, the key assumption 
is that these maximal functions must be bounded ``near" $L^2.$ As mentioned before, V. Kova\v{c}'s result \cite{Kovac1} has a seemingly different assumption on the measure. For his purposes, 
it is important that the measure is finite -- implied by the fact that the measure is complex -- and that the gradient of its Fourier transform satisfies a decay of the type 
\[
|\nabla \widehat{\mu}(\xi)| \le C (1+ |\xi|)^{-1-\eta}, \, \text{ for some } \eta >0.
\]
The next proposition shows that Kova\v{c}'s hypotheses actually \emph{imply} ours. We mention that this result is far from new, with s similar version appearing in \cite{SteinSogge}. For the convenience of the reader, we quickly review the results from \cite{RubioDeFrancia}: 

\begin{prop}\label{multipl} Let $T^*f(x) = \sup_{t>0} |\mathcal{F}^{-1} (m(t\cdot) \widehat{f})|.$ Suppose that 
\[
|m(\xi)| \lesssim (1+|\xi|)^{-a}, \, |\nabla m(\xi)| \lesssim (1+|\xi|)^{-b}, 
\]
with $a+b >1.$ Then $T^* : L^2(\R^n) \to L^2(\R^n)$ boundedly.
\end{prop}

\begin{proof} Letting $\psi_0: \R^n \to \R$ be a (radial) smooth function supported in the annulus $\{y \colon 1/2 \le |y| \le 2\}$ so that 
\[
\sum_{j \in \Z} \psi_0(2^j \xi) = 1, \, \forall \xi \ne 0,
\]
we define $m_j(\xi) := m(\xi) \psi_0(2^j \xi).$ By letting $T^*_j$ denote the maximal multiplier operator associated to each of these multipliers, we have 
\[
T^* f \le T^*_0 f + \sum_{j \le 0} T^*_j f. 
\]
Here, we let $\sum_{j > 0} m_j(\xi) = \phi_0(\xi)$ and define the operator $T^*_0$ to be the maximal multiplier operator associated to $\phi_0.$ As $\phi_0$ is a smooth function with compact support, this operator is bounded pointwise by a maximal function. 
We then move on to estimate each factor $T^*_jf$ individually: we bound the supremum by
\begin{align*}
\sup_{t>0} |\mathcal{F}^{-1} (m_j(t\cdot)\widehat{f})(x)|^2 &  \le \left( \int_0^{\infty}| T_{j,t}f(x) \cdot \tilde{T}_{j,t}f(x)| \, \frac{\mmd t}{t} \right), \cr
\end{align*}
where $\widehat{T_{j,t}f}(\xi) = m_j(t \xi) \widehat{f}(\xi),\,\, \widehat{\tilde{T}_{j,t}f}(\xi) = \tilde{m}_j(t\xi) \widehat{f}(\xi),$ with $\tilde{m}_j(\xi) = \xi \cdot \nabla m_j (\xi).$ We estimate then 

\begin{align*}
\|T_j^*f\|_2^2 & = \left( \int_0^{\infty} \int_{\R^n} |m_j(t\xi)\widehat{f}(\xi)|^2 \, \mmd \xi \, \frac{\mmd t}{t} \right)^{1/2} \times \left( \int_0^{\infty} \int_{\R^n} |\tilde{m}_j(t\xi)\widehat{f}(\xi)|^2 \, \mmd \xi \, \frac{\mmd t}{t} \right). \cr 
\end{align*}
The integrals above exist only for $2^j t |\xi| \in [1/2,2].$ Therefore, using the decay properties of $m, \tilde{m}$, we obtain
\[
\|T_j^*f\|_2^2 \lesssim 2^{ja} 2^{j(b-1)} \|f\|_2^2 = 2^{j(a+b-1)} \|f\|_2^2.
\]
As we supposed that $a+b > 1,$ the series above is summable in $j<0,$ which completes the proof. 
\end{proof}

Theorem \ref{meastheo1} not only recovers a version of the two-dimensional results from Kova\v{c}, but also allows us to extend them, as mentioned before, to a larger class of maximal functions. For instance, 
Bourgain's circular maximal function fulfills the conditions to Theorem \ref{meastheo1}, whereas the gradient 
$$ |\nabla \widehat{\sigma_{\mathbb{S}^1}}(\xi)|  \sim |\xi|^{-1/2}$$
for non-trivial sets of $|\xi| \to \infty$ in two dimensions, so that Kova\v{c}'s result does not apply. Also, the spherical maximal function in dimension three satisfies that 
$$ |\nabla \widehat{\sigma_{\mathbb{S}^2}}(\eta)| \sim |\eta|^{-1}$$ 
on a non-trivial set of $|\eta| \to \infty$, but, as $|\widehat{\sigma_{\mathbb{S}^2}}(\eta)| = O(|\eta|^{-1}),$ it is still possible to use Proposition \ref{multipl} to conclude the $L^2-$boundedness of this operator, which is all we need to conclude. 

\subsection{The spherical maximal functions and previous maximal restriction results} In \cite{Ramos1}, this author proves a full range 2-dimensional maximal restriction estimate for the \emph{strong maximal function}. Namely,
the main theorem there is that 
\[
\|M_{\mathcal{S}}(\widehat{f})\|_{L^q(\mathbb{S}^1)} \lesssim_p \|f\|_{L^p(\R^2)},
\]
with $M_{\mathcal{S}}g(x) = \sup_{\substack{R \text{ axis parallel}, \\ \text{centered at} x}} \dashint_R |g|.$ One might ask is whether Theorem \ref{meastheo1} implies the result above through a pointwise domination, as 
the spherical maximal function dominates the usual Hardy--Littlewood maximal function. Our next result shows that the answer is \emph{no} in all dimensions larger than 1.

\begin{prop} Let $d \ge 2.$ Then there exists $f \in \cap_{p \ge 1} L^p(\R^d)$ such that 
\[
\text{\emph{ess sup}}_{x \in \R^d} \frac{M_{\mathcal{S}}f(x)}{M_{\mathbb{S}^{d-1}}f(x)} = + \infty.
\]
\end{prop} 

\begin{proof} Let first $d \ge 3.$ In these cases, the counterexample is much simpler. In fact, we take $f = \chi_{Q(0,1)}$, the characteristic of the unit cube. It is a simple calculation to 
verify that $M_{\mathcal{S}}f(x) \gtrsim \frac{1}{|x|}$ whenever $|x| \gg 1.$ Also, one obtains in a fairly straightforward manner that $M_{\mathbb{S}^{d-1}}f(x) \lesssim \frac{1}{|x|^{d-1}}, \, |x| \gg 1.$ As $d-1 > 1,$ $f$ is a sought-after
counterexample. \\

In dimension $d=2$ matters are subtler. Let $g_n(x_1,x_2) = \chi_{[0,1]}(x_1) \chi_{[0,1/n^{20}]}(x_2).$ We take a sequence $(y_n,r_n) \in \R^2 \times \R_{+}$ such that
\begin{itemize}
\item $r_{n+1} = 10^n r_n, \, r_1 = 1;$
\item $y_{n+1} = (r_1 + 2(r_2 + \cdots + r_n) + r_{n+1},0).$
\end{itemize}
We then set up the function $f(x) = \sum_{n =1}^{\infty} g_n(x-y_n).$ This function is clearly in any $L^p$ space. We estimate the strong and spherical maximal functions for $x$ in a strip near $y_n.$ \\ 

Effectively, let $x \in S_n := y_n + [-10^n,10^n] \times [0, 1/n^{20}]$. Similarly as in the high dimensional case, $M_{\mathcal{S}}f(x) \gtrsim \frac{1}{|x-y_n|}.$ Now we split the spherical maximal 
function into two parts as 
\begin{equation}\label{splitt}
M_{\mathbb{S}^1}f = \text{max} \{M_{\mathbb{S}^1, \ge r_{n}}f, M_{\mathbb{S}^1, < r_{n}}f\}.
\end{equation}
Here, $M_{\mathbb{S}^1, \ge t}g$ stands for the maximal function obtained by only taking radii larger than $t,$ and define analogously $M_{\mathbb{S}^1, < t}f$. By the properties of the radii $r_n$ and the way we defined $y_n,$ 
\[
M_{\mathbb{S}^1, \ge r_{n}}f(x) \lesssim \frac{1}{r_n}.
\]
Also, for the local part we obtain
\[
M_{\mathbb{S}^1, < r_{n}}f(x) \lesssim \frac{1}{n^{10} |x-y_n|}.
\]
Substituting these inequalities in the quotient, using \eqref{splitt}, we get 
\[
\frac{M_{\mathcal{S}}f(x)}{M_{\mathbb{S}^1}f(x)} \gtrsim \text{min}\left\{ n^{10}, \frac{r_n}{|x-y_n|}\right\}.
\]
Notice that $r_n = 10^{\frac{n(n-1)}{2}}$ and that $|x-y_n| \lesssim 10^n$ in $S_n.$ We have found a set of measure $\gtrsim 10^{n/2}$ where the desired quotient is at least $n^{10}.$ But these sets are mutually disjoint, 
which readily implies that the $L^{\infty}$ norm of the quotient is not finite.
\end{proof} 

\subsection{Theorems \ref{rtheo1} and \ref{rtheo2} and a Knapp-like counterexample}\label{knapp} In this Section, we adapt the classical Knapp counterexample to obtain constraints on $s$, in order for versions of Theorems \ref{rtheo1} and \ref{rtheo2} to hold for a family of strong maximal functions:

\begin{prop} Let 
\[
M_{\mathcal{S},s}\, g = \left(\sup_{\substack{R \text{ axis parallel}, \\ \text{centered at } x}} \dashint_R |g|^s \right)^{1/s}
\]
denote the $s-$strong maximal function, in either two or three dimensions. Suppose that 
\[
\| M_{\mathcal{S},s} \, \widehat{g} \|_{L^q(\mathbb{S}^1)} \lesssim \|g\|_{L^p(\R^2)},
\]
whenever $1 \le p < \frac{4}{3}$ and $3q \le p'.$ Then $s \le 4.$ \\ 

Analogously, suppose that 
\[
\| M_{\mathcal{S},s} \, \widehat{g} \|_{L^2(\mathbb{S}^2)} \lesssim \|g\|_{L^p(\R^3)},
\]
for all $1 \le p \le 4/3.$ Then $s \le 2.$ 
\end{prop}

Before we move on to the proof, we remark that a combination of the proofs of Theorems \ref{rtheo1}, \ref{rtheo2} and the ideas in \cite{Ramos1} attains that 
\[
\|M_{\mathcal{S},s} \widehat{f}\|_{L^q(\mathbb{S}^1)} \lesssim \|f\|_{L^p(\R^2)}, \,\text{whenever } 1 \le p < \frac{4}{3}, p'\ge 3q \text{ and } s \le 2,
\]
and 
\[
\|M_{\mathcal{S},s}\widehat{g}\|_{L^2(\mathbb{S}^2)} \lesssim \|g\|_{L^p(\R^3)}, \, \text{whenever } 1 \le p \le \frac{4}{3} \text{ and } s < 2.
\]
We spare the details, for their proofs are essentially the same as the ones presented. 

\begin{proof} We begin with the two-dimensional part. Let $\widehat{f_t}(\xi_1,\xi_2) = \chi_{(-t,t)}(\xi_1) \chi_{(1-t^2,1)}(\xi_2).$ We call this the \emph{box-Knapp example}. It is easy to compute that 

\[
\| f_t \|_{L^p(\R^2)} = C \cdot t^{3 - 3/p}, \, \forall t >0.
\]
On the other hand, we estimate the maximal function $M_{\mathcal{S},s} (\widehat{f_t})$ from bellow as follows. Fix a small angle $\theta_0 >0.$ Then, for $\theta \in (\pi/4, \pi/2 - \theta_0),$
there is a constant $c(\theta_0)$ so that $\cos(\theta), 1- \sin(\theta) \ge c(\theta_0).$ We estimate:
\begin{align*}
M_{\mathcal{S},s}(\widehat{f_t})(e^{i \theta}) & \ge \left( \dashint_{(-t,\cos(\theta)) \times (\sin(\theta),1)}  \chi_{(-t,t) \times (1-t^2,t)} \right)^{1/s} \gtrsim \frac{t^{3/s}}{\cos(\theta)^{1/s} (1-\sin(\theta))^{1/s}}  \gtrsim t^{3/s}. \cr
\end{align*}
This is the estimate we need, for then
\[
\|M_{\mathcal{S},s}(\widehat{f_t})\|_{L^q(\mathbb{S}^1)} \gtrsim \left(\int_{\pi/4}^{\pi/2-\theta_0} t^{3q/s} \mmd \theta\right)^{1/q} \gtrsim_{\theta_0} t^{3/s}. 
\]
Putting together yields that 
\[
\forall\, 0 < t \ll 1, \, t^{3/s} \lesssim t^{3-3/p} \iff \frac{1}{s} + \frac{1}{p} - 1 \ge 0.
\]
If $s = 4 + \varepsilon,$ then $1/p \ge \frac{3 + \varepsilon}{4+ \varepsilon} \iff p \le \frac{4+\varepsilon}{3 + \varepsilon} < \frac{4}{3},$ and the restriction estimates cannot hold in the full two-dimensional range. \\ 

For the three-dimensional part, we let $\widehat{F_t} (\eta_1,\eta_2,\eta_3) = \chi_{B^2(0,t)}(\eta_1,\eta_2) \chi_{(-1,1)}(\eta_3),$ and call this a \emph{long-Knapp example}. Again, a computation shows that
\[
\| F_t \|_{L^p(\R^3)} = \tilde{C} t^{2 - 2/p}, \, \forall t>0.
\]
In this case, we bound $M_{\mathcal{S},s} (\widehat{F_t})$ from below by the $s-$average over a rectangle of dimensions $t \times t \times 4$ centered at each point $x \in \mathbb{S}^2$. In a spherical region of positive $\mathcal{H}^2-$measure, we have 
\[
M_{\mathcal{S},s}(\widehat{F_t}) \gtrsim t^{1/s} \Rightarrow t^{1/s} \lesssim t^{2-2/p}, \, \forall \, t \text{ small} \iff \frac{1}{s} - 2 + \frac{2}{p} > 0.
\]
Again, if $s >2,$ then $p$ is forced to be strictly less than $4/3.$ 
\end{proof}

With the long-Knapp example, we prove the following:

\begin{prop}\label{sharpp} The only dimensions in which maximal restriction estimates for $M_{\mathcal{S}} := M_{\mathcal{S},1}$ can hold in the full range are $d=2,3.$ 
\end{prop} 

\begin{proof} By an argument using long-Knapp example from above, in order for the full range of maximal restriction estimates of the kind 
\begin{equation}\label{sharprest}
\|M_{\mathcal{S},s} (\widehat{f}) \|_{L^q(\mathbb{S}^{d-1})} \lesssim \|f\|_{L^p(\R^d)}
\end{equation}
to hold in the same regime as the already known restriction estimates, we must have $s \le \frac{2(n+1)}{(n-1)^2}.$ This number is less than $1$ if $n \ge 5.$ Also, using the results from \cite{Guth2} (see also \cite{RogHickman} for further developments), 
we know that the restriction estimates from \ref{restriction1} in dimension 4 for the sphere hold as long as $p' > 2.8.$ Thus, in order for \ref{sharprest} to hold in the full range for $d=4,$ we need $s \le \frac{2.8}{3} < 1.$ 
In particular, this implies that $M_{\mathcal{S}}$ cannot be bounded in the full range, except for when $d=2$ or $d=3.$ 
\end{proof}

As proved in \cite{Ramos1}, these estimates do hold in the case of the two-dimensional problem. An interesting question is the validity of the same bounds in dimension 3. Nevertheless, an affirmative answer would trivially imply the 
three-dimensional restriction conjecture, which is still not completely settled. \\ 

Note that the long-Knapp example, if translated to 2 dimensions, provides us with the \emph{exact same} bounds as we have achieved. In fact, one achieves that, for $\widehat{\tilde{f_t}}(\xi_1,\xi_2) = \chi_{(-t,t)}(\xi_1) \chi_{[-1,1]}(\xi_2),$
\[
t^{1/s} \lesssim \|M_{\mathcal{S},s}(\widehat{\tilde{f}_t}) \|_{L^q(\mathbb{S}^1)} \lesssim \|\tilde{f}_t\|_{L^p(\R^2)} \lesssim t^{1-1/p} \iff p' \ge s \Rightarrow s \le 4.
\]
Thus, we get no improvement from changing the counterexample's nature. Furthermore, if we replace the strong maximal function by the Hardy--Littlewood maximal function in \emph{any} dimension, the long-Knapp and the box-Knapp examples deliver the same 
bounds for $s$: 
\[
\frac{1}{s} - \frac{1}{p'} \ge 0 \iff s \le p'.
\]
For the three-dimensional Tomas-Stein exponent case, we get the same $s \le 4$ bound as in the two dimensions. One inquires whether there is any fundamental difference between the strong and the Hardy--Littlewood maximal functions in this context. Our counterexamples
seem to hint at an intrinsic geometric distinction. \\ 

The three-dimensional Theorem \ref{rtheo2} is essentially sharp, in the sense that we have attained an almost exact characterization of when the full range maximal restriction estimates work. The only remaining case is the $s=2, p=4/3$ case. We suspect that the inequality should fail in that endpoint. At the moment, we
can neither prove nor disprove it.

\end{document}